\documentclass[letterpaper, 10 pt, conference]{ieeeconf}  

\IEEEoverridecommandlockouts                              
\usepackage{graphics} 
\usepackage{epsfig} 
\usepackage{mathptmx} 
\usepackage{times} 

\usepackage{amsmath,amsfonts,amssymb}
\newtheorem{theorem}{Theorem}
\newtheorem{lemma}{Lemma}
\newtheorem{problem}{Problem}

\usepackage{color}

\newcommand{\ikd}[1]{#1}
\newcommand{\argmax}{\mathop{\rm arg~max~}\limits}
\newtheorem{assumption}{Assumption}

\title{\LARGE \bf
Switched Systems Control via Discreteness-Promoting Regularization
}

\author{Masaaki Nagahara, Takuya Ikeda, Ritsuki Hoshimoto
\thanks{This work was partly supported by 
JSPS KAKENHI Grant Numbers JP23K26130, JP22H00512,  JP22KK0155, JP24K21314, 24K17300, and also JST ASPIRE Program Grant Number JPMJAP2402.
}
\thanks{M. Nagahara and R. Hoshimono are with Hiroshima University, Hiroshima, Japan. {\tt nagahara@ieee.org}, {\tt ritsuki.2002.730@icloud.com}}%
\thanks{T. Ikeda is with The University of Kitakyushu, Fukuoka, Japan.
{\tt t-ikeda@kitakyu-u.ac.jp}}%
}

\begin{document}

\maketitle
\thispagestyle{empty}
\pagestyle{empty}

\begin{abstract}
This paper proposes a novel method for designing finite-horizon discrete-valued switching signals in linear switched systems based on discreteness-promoting regularization. The inherent combinatorial optimization problem is reformulated as a continuous optimization problem with a non-convex regularization term that promotes discreteness of the control. We prove that any solution obtained from the relaxed problem is also a solution to the original problem. The resulting non-convex optimization problem is efficiently solved through time discretization. Numerical examples demonstrate the effectiveness of the proposed method.
\end{abstract}

\section{INTRODUCTION}
In this paper, we explore finite-horizon control of linear switched systems, characterized by two or more linear subsystems governed by a switching signal \cite{LibMor99,LinAnt09}.
Real applications of switched systems can be found in hybrid electric vehicles \cite{HeTiaNiXu17}, autonomous vehicles \cite{ZhaIfgPui24}, 
semiactive suspensions \cite{GerColBol08}, power systems \cite{NogKoo20}, and chemical processes \cite{NiuZhaFanChe15}.

Mathematically, the design of switching signals is the problem of \emph{discrete-valued} control. For example, any switching signal between two systems should be binary. Hence, when we consider the optimal control of switched systems, the problem is naturally formulated as mixed-integer programming \cite{ZhuAnt15}, which is hard to solve in real time.

On the other hand, the paper \cite{IkeNagOno17} has proposed a convex optimization approach to discrete-valued control of linear time-invariant systems. This method is based on the sum-of-absolute-values (SOAV) optimization \cite{Nag15}, which is formulated as the sum of $L^1$ norms of biased control signals.
In this paper, we adapt this method to the optimal design of switching signals of linear switched systems. Specifically, we reformulate the design as an optimal control problem with the \emph{one-hot} control, containing all zero elements except for a single one that represents the current switching mode.

In this representation, the $L^1$ norm of the one-hot control is always constant, and hence we cannot use the $L^1$ norm-based SOAV optimization. Instead, we propose to use non-convex cost functions \cite{Ike24}, such as the $L^1-L^2$ ($L^1$ minus $L^2$) norm \cite{YinLouHeXin15}, the $L^p$ norm with $0<p<1$ \cite{IkeKas19}, and the minimax concave penalty \cite{HayIkeNag24}, to promote the one-hot property of the optimal control.
Under mild assumptions on the regularization term, we show that any solution to the relaxed problem is one-hot, and also a solution to the original problem. Moreover, we propose an efficient numerical optimization method based on time discretization \cite{Nag20}, or  multiple shooting \cite{DieBocDieWie06}, which can be solved by the DC (difference of convex functions) programming \cite{Ike24},
or by using numerical optimization packages such as \texttt{cvxpy} \cite{DiaBoy16,AgrVerDiaBoy18} and \texttt{casadi} \cite{Casadi}.

We demonstrate numerical examples of switching control of second-order two- and three-mode unstable linear systems. We implement the proposed finite-horizon control to construct feedback control by the model predictive control strategy. In these examples, we show that the designed switching signals successfully stabilize the unstable systems.

\subsection*{Notation}
For a vector $x=[x_1,x_2,\ldots,x_N]^\top\in\mathbb{R}^N$, the $\ell^p$ norm is defined by
\begin{equation}
    \|x\|_p \triangleq \left(\sum_{i=1}^N |x_i|^p\right)^{1/p}.
\end{equation}
For two vectors $x=[x_1,x_2,\ldots,x_N]^\top$ and
$y=[y_1,y_2,\ldots,y_N]^\top$, the Hadamard product is denoted by $x\odot y$, that is,
\begin{equation}
    x\odot y \triangleq 
    [x_1y_1,  x_2y_2,  \ldots,  x_Ny_N]^\top.
\end{equation}
For a function $u:[0,T]\rightarrow \mathbb{R}^N$, the $L^p$ norm is defined by
\begin{equation}
    \|u\|_{L^p} \triangleq \left(\int_0^T \|u(t)\|_p^p dt\right)^{1/p}.
\end{equation}

\section{MATHEMATICAL FORMULATION}
Let us consider the following linear switched system:
\begin{equation}
    \dot{x}(t) = A_{\sigma(t)} x(t),\quad t\geq 0, \quad x(0)=\xi\in\mathbb{R}^n,
    \label{eq:plant}
\end{equation}
where $\sigma(t)\in\{1,2,\ldots,N\}$ is the switching signal to be designed, and $A_i\in\mathbb{R}^{n\times n}$, $i\in\{1,2,\ldots,N\}$ are fixed matrices.
The control purpose is to design the switching signal $\sigma(t)$ over a finite horizon $[0,T]$ with fixed $T>0$ such that the state $x(t)$ approaches the origin.
We should note that there is generally no $\sigma(t)$ over $[0,T]$ that achieves $x(T)=0$ for any finite $T>0$ since there is no control input to \eqref{eq:plant}.
Therefore, we instead minimize $x(T)^\top Qx(T)$ with positive definite $Q$. Namely, we consider the following optimal control problem:
\begin{problem}\label{prob:original}
    For the switched system \eqref{eq:plant}, find the switching signal $\sigma(t)\in\{1,2,\ldots,N\}$ for  $t\in[0,T]$ that minimizes
    \begin{equation}\label{eq:cost_original}
        J(u) = x(T)^\top Q x(T).
    \end{equation}
\end{problem}

\section{\ikd{RELAXATION TO CONTINUOUS OPTIMIZATION}} 
In this section, we propose relaxation of Problem \ref{prob:original} to continuous optimization that can be  efficiently solved by numerical optimization that will be discussed in Section \ref{sec:numerical computation}.
We also show that the solution to the relaxed optimization problem is also a solution to the original Problem \ref{prob:original}.

\subsection{Relaxed Problem}
To obtain the optimal control numerically, we transform the system \eqref{eq:plant} into the following system:
\begin{equation}
    \dot{x}(t) = \sum_{i=1}^N A_i x(t) u_i(t)
    \label{eq:plant_transformed}
\end{equation}
where $u_i(t)\in\{0,1\}$, $i=1,2,\ldots,N$ are the new control variables defined as
\begin{equation}
    u_i(t) = \begin{cases} 
                1, \text{~if~} \sigma(t)=i,\\
                0, \text{~otherwise.}
             \end{cases}   
\end{equation}  
The plant model \eqref{eq:plant_transformed} is represented as
\begin{equation}
    \dot{x}(t) = \Phi\bigl(x(t)\bigr)u(t),
\end{equation}
where
\begin{equation}
    \Phi(x) \triangleq \begin{bmatrix}A_1x,A_2x,\ldots,A_Nx\end{bmatrix},\quad
    u(t) \triangleq \begin{bmatrix}u_1(t)\\u_2(t)\\ \vdots \\ u_N(t)\end{bmatrix}.
    \label{eq:plant_transformed_vector}
\end{equation}
The vector $u(t)$ is so-called a \emph{one-hot} vector that satisfies
\begin{equation}
    u_i(t) \in \{0,1\},\quad \sum_{i=1}^N u_i(t) = 1,\quad
    i\in \{1,2,\ldots,N\}
\label{eq:original constraints}
\end{equation}
for almost all $t\in[0,T]$.
From \eqref{eq:plant_transformed_vector}, the linear switched system in \eqref{eq:plant} is described as a control-affine system with $N$-dimensional control input.

Then, the original problem (Problem \ref{prob:original}) is equivalently translated as follows:
\begin{problem}\label{prob:original2}
    For the system \eqref{eq:plant_transformed}, find the control $u(t)\in\{0,1\}^N$, $t\in[0,T]$ that \ikd{satisfies the constraints in~\eqref{eq:original constraints} and} minimizes
    $J(u) = x(T)^\top Q x(T)$.
\end{problem}

This problem is a discrete-valued control problem.
This type of problem can be described as a mixed-integer program, which often has a heavy computational burden.
Instead, we adopt the following relaxed constraint:
\begin{equation} 
    u_i(t) \in [0,1],\quad \sum_{i=1}^N u_i(t) = 1,\quad
    i\in \{1,2,\ldots,N\}.
\label{eq:constraints}
\end{equation}
At the same time, to promote the discreteness property $u_i(t)\in\{0,1\}$ in \eqref{eq:original constraints}, we add a regularization term $\Omega(u)$ to the cost function $J(u)$. 
Namely, we consider the following regularized cost function:
\begin{equation}
    J_\lambda (u) \triangleq x(T)^\top Qx(T) + \lambda \Omega(u),
        \label{eq:cost}
\end{equation}
where $\lambda>0$ is the regularization hyper-parameter.
In particular, we adopt $\Omega(u)$ in the form of
\begin{equation}
     \Omega(u)\triangleq\int_0^T \psi(u(t)) dt,
     \label{eq:Omega}
\end{equation}
where $\psi$ is a function.

Now, we summarize the relaxed problem as follows.
\begin{problem}\label{prob:relaxed_ver2}
    Find the control $u(t)\in\mathbb{R}^N$, $t\in[0,T]$ that solves the following optimal control problem:
\begin{equation}
\begin{aligned}
  & \underset{u}{\text{minimize}}
  & & x(T)^\top Q x(T) + \lambda \int_0^T \psi(u(t)) dt\\
  & \text{subject to}
  & & \dot{x}(t) = \Phi\bigl(x(t)\bigr)u(t), \quad x(0)=\xi,\\
  & & & u(t)\in\mathcal{U} \mbox{~for~almost~all~} t\in[0, T],
\end{aligned}
\end{equation}
where $\mathcal{U}$ is defined by
\begin{equation}
    \mathcal{U} \triangleq
    \left\{ u\in\mathbb{R}^{N}:
    u_i \in [0, 1],~\sum_{i=1}^N u_i = 1
    \right\}.
\end{equation} 
\end{problem}

\subsection{Functions that promote discreteness}
An important problem is to find an appropriate function $\psi(u)$ with which any solution to Problem \ref{prob:relaxed_ver2} is also a solution to the original Problem \ref{prob:original2}.

First, we should note that the convex optimization approach by the sum of absolute values \cite{IkeNagOno17} cannot be adopted here. 
In fact, the regularization term by this is described as
\begin{equation}\begin{split}
    \Omega(u) &= \int_0^T \left(\|u(t)\|_1 + \|1-u(t)\|_1\right) dt\\ &
    = \int_0^T \sum_{i=1}^{N} \left(|u_i(t)| + |1-u_i(t)|\right) dt,
\end{split}\label{eq:regularization term SOAV}
\end{equation}
but it follows from the constraints in \eqref{eq:constraints} that $\Omega(u)$ is constant for all $u$ satisfying \eqref{eq:constraints}.

Instead, as an example, we can intuitively adopt the following regularization term:
\begin{equation}
\begin{split}
     \Omega(u)&=\int_0^T \left\|u(t)\odot \bigl(1-u(t)\bigr)\right\|_1 dt\\
     &= \int_0^T \sum_{i=1}^N \left|u_i(t)\bigl(1-u_i(t)\bigr)\right| dt.
\end{split}     
     \label{eq:regularization term}
\end{equation}
This $L^1$-norm regularization term forces $u_i(t)$ to take values of 0 and 1 for significant time durations \cite{Nag15,IkeNagOno17}.
Figure \ref{fig:psi} shows the graph of function $\psi_i(u)=u(1-u)$. We can see that $\psi_i(0)=\psi_i(1)=0$, which promotes the discreteness property in \eqref{eq:original constraints} as shown in the next subsection.
\begin{figure}[t]
 \includegraphics[width=\linewidth]{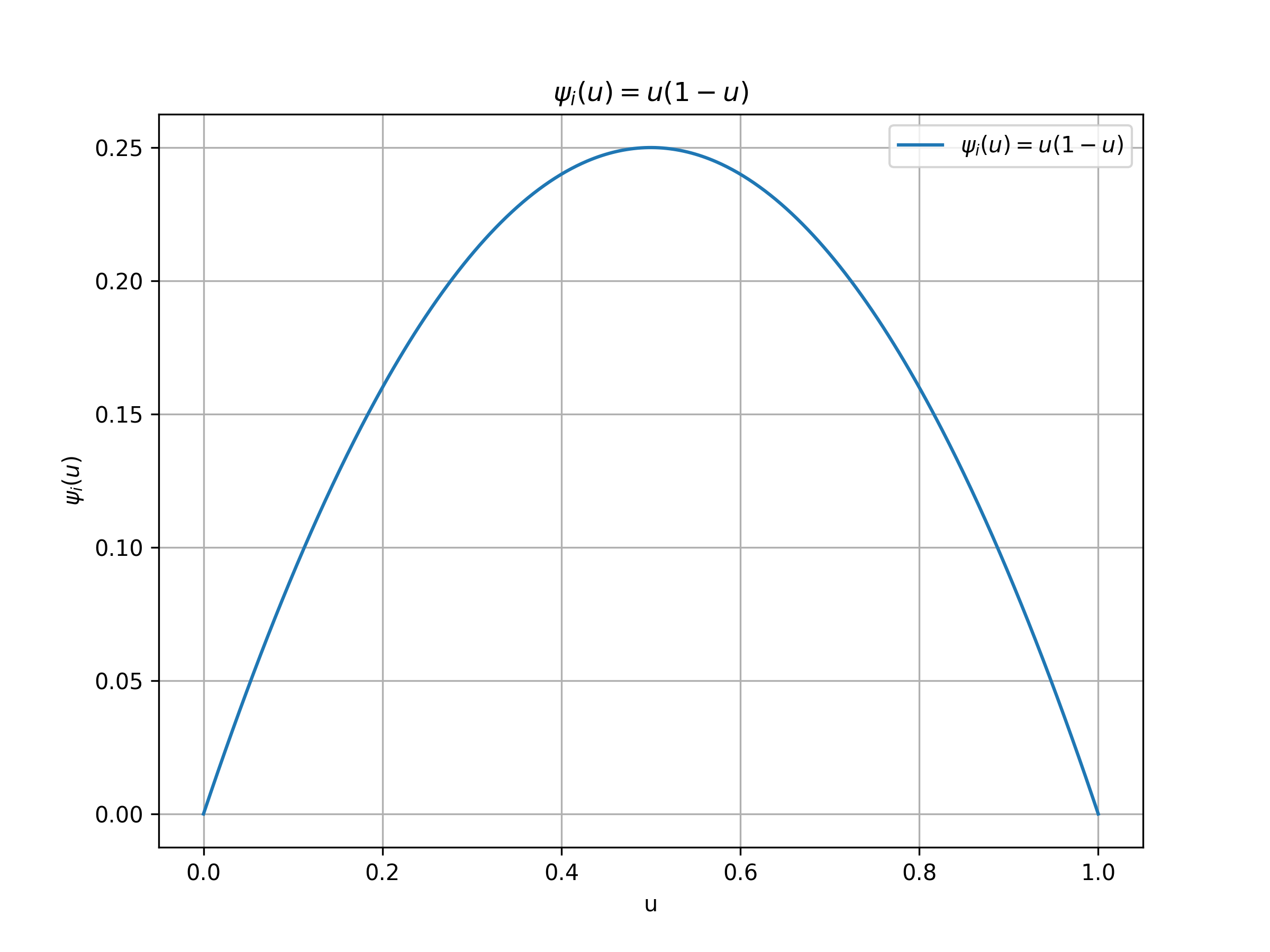}
 \caption{The graph of $\psi_i(u)=u(1-u)$.}
 \label{fig:psi}
\end{figure}
Additionally, since $u_i(t)\in[0,1]$ for any $t$ and $i$ from \eqref{eq:constraints}, 
the regularization term \eqref{eq:regularization term} can be rewritten as
\begin{equation}
\begin{split}
    \Omega(u) &= \int_0^T \sum_i^N u_i(t) dt - \int_0^T \sum_{i}^N u_i(t)^2 dt\\
    &= \|u\|_{L^1} - \|u\|_{L^2}^2.
\end{split}    
\end{equation}
The function $\|u\|_{L^1} - \|u\|_{L^2}^2$ is called the $L^1-L^2$ ($L^1$ minus $L^2$) function, known to promote  \emph{sparsity} of $u$ \cite{YinLouHeXin15}. 
Meanwhile, the set of $u_i(t)$ for each $t$ satisfying the constraints in \eqref{eq:constraints} is the \emph{convex hul} of the points defined in \eqref{eq:original constraints}, which is the $(N-1)$-simplex in $\mathbb{R}^N$.
These relaxations yield a tractable, non-convex optimization problem known as difference-of-convex (DC) programming \cite{Ike24}.

Now, we propose functions that promote discreteness like the $L_1-L_2$ function. Such functions are characterized as follows.
\begin{assumption}\label{ass:psi}
    The function $\psi:\mathbb{R}^N\to\mathbb{R}$ satisfies the following:
    \begin{enumerate}
        \item $\psi$ is additively separable, i.e., there exist measurable functions $\psi_i:\mathbb{R}\to\mathbb{R}$, $i=1,2,\dots,N$, that satisfy 
        \begin{equation}
            \psi(u) = \sum_{i=1}^{N} \psi_i(u_i),
        \end{equation}
        where $u = [u_1, u_2, \dots, u_N]^\top$.
        \item $\psi_i(0) = \psi_i(1) = 0$ for all $i\in\{1,2,\dots,N\}$.
        \item $\psi_i(u_i) > 0$ on $(0, 1)$ for all $i\in\{1,2,\dots,N\}$.
    \end{enumerate}
\end{assumption}
We note that $\psi_i(u)=u(1-u)$ in \eqref{eq:regularization term} satisfies this assumption, whereas $\psi_i(u)=|u|+|1-u|$ in \eqref{eq:regularization term SOAV} and its shifted version $\psi_i(u)=|u|+|1-u|-1$ do not.

In the next section, we show 
that any solution to the relaxed Problem \ref{prob:relaxed_ver2} is a solution to the original Problem \ref{prob:original2} under this assumption.

\subsection{Validity theorem for the relaxed problem}
First, we show that the optimal control for Problem \ref{prob:relaxed_ver2} takes discrete values of 0 and 1.
For this, we prepare the following lemma.

\begin{lemma}\label{lem:max_cond}
Suppose Assumption~\ref{ass:psi} \ikd{holds}, and consider the following optimization problem:
\begin{equation}\label{prob:lemma}
\begin{aligned}
  & \underset{u\in\mathcal{U}}{\text{maximize}}
  & & \sum_{i=1}^{N} \left(\rho_i u_i - \lambda \psi_i(u_i)\right),
\end{aligned}
\end{equation}
where $\lambda>0$,
and 
$\rho_1, \rho_2, \dots, \rho_N$ are given real numbers satisfying
\begin{equation}
    \rho_{1} = \cdots = \rho_{m} > \rho_{m+1} \geq \cdots \geq \rho_{N}
\end{equation}
for some $m\in\{1,2,\dots,N\}$.
Then, any optimal solution $u^\ast$ to problem~\eqref{prob:lemma} satisfies $u^\ast\in\{e_1,\dots,e_m\}$, 
where $e_i$ is the $i$th canonical vector in $\mathbb{R}^N$.
\end{lemma}
\begin{proof}
Let us denote by $f$ the cost function in problem~\eqref{prob:lemma}, and take any optimal solution $u^\ast$, i.e.,
$f(u^\ast) \geq f(u)$ for all $u\in\mathcal{U}$.
Note that, for any $u\in\mathcal{U}$, we have 
\begin{equation}\label{eq:lemma_cost}
    f(u) = \rho_1 + \sum_{i=m+1}^{N} \left(\rho_i - \rho_1\right)u_i - \lambda \psi(u).
\end{equation}
Hence, we have
\begin{equation}\label{eq:lemma_cost2}
    0 \leq f(u^\ast) - f(e_1) = \sum_{i=m+1}^{N} \left(\rho_i - \rho_1\right)u_i^\ast - \lambda \psi(u^\ast) \leq 0,
\end{equation}
where 
the first relation follows from $e_1\in\mathcal{U}$ and the optimality of $u^\ast$,
the second relation follows from \eqref{eq:lemma_cost} and $\rho_1 = f(e_1)$,
and the third relation follows from Assumption~\ref{ass:psi}.
It follows from~\eqref{eq:lemma_cost2} that we have
$u_{i}^\ast = 0$ for all $i \geq m+1$
and $\psi_i(u_i^\ast) = 0$ for all $i \leq m$.
This completes the proof.
\end{proof}

By this lemma, we show the discreteness of the solution(s) to Problem \ref{prob:relaxed_ver2}.
\begin{lemma}\label{lemma:discreteness}
    Suppose Assumption~\ref{ass:psi} \ikd{holds}.
    Any optimal solution $u^*$ to Problem \ref{prob:relaxed_ver2} satisfies the \ikd{constraints in} \eqref{eq:original constraints} for almost all $t\in[0,T]$.
\end{lemma}
\begin{proof}
Let us take any solution $u^\ast$ to Problem~\ref{prob:relaxed_ver2} and denote by $x^\ast$ the corresponding state.
It follows from Theorem~22.2 and Corollary~22.3 in~\cite{Cla} that there exists a function $p: [0, T]\to \mathbb{R}^n$ that satisfies
\begin{equation}
    u^\ast(t) \in \argmax_{u\in \mathcal{U}} \sum_{i=1}^{N} \left(p(t)^\top A_i x^\ast(t) u_i - \lambda \psi_i(u_i)\right)
\end{equation}
almost everywhere.
Put $\rho_i(t)\triangleq p(t)^\top A_i x^\ast(t)$, 
and define the index $i_k(t)\in\{1,2,\dots,N\}$ so that $i_k(t) \neq i_l(t)$ for any $k \neq l$ and $\rho_{i_k(t)}(t)$ is arranged in descending order, i.e.,
\begin{equation}
    \rho_{i_1(t)}(t) \geq \rho_{i_2(t)}(t) \geq \cdots \geq \rho_{i_N(t)}(t).
\end{equation}
(Note that the index $i_k(t)$ is not uniquely determined in general; however, this does not affect the result.)
Then, it follows from Lemma~\ref{lem:max_cond} that we have
\begin{equation}
    u^\ast(t) \in \{e_{i_1(t)}, \dots, e_{i_m(t)}\},
\end{equation}
where $m\in\{1,2,\dots,N\}$ is the number satisfying 
\begin{equation}
    \rho_{i_1(t)}(t) = \cdots = \rho_{i_m(t)}(t) > \rho_{i_{m+1}(t)}(t) \geq \cdots \geq \rho_{i_N(t)}(t).
\end{equation}
This completes the proof.
\end{proof}

Finally, we have the following theorem that validates the proposed  relaxation.
\begin{theorem}
    Suppose Assumption \ref{ass:psi} holds. Any optimal solution $u^*$ to Problem \ref{prob:relaxed_ver2} is a solution to Problem \ref{prob:original2}.
\end{theorem}
\begin{proof}
    Let $u^*$ and $v^*$ be optimal solutions to Problem \ref{prob:relaxed_ver2} and \ref{prob:original2}, respectively. From Lemma \ref{lemma:discreteness}, $u^*(t)$ satisfies \eqref{eq:original constraints} almost everywhere. Also, since $v^*$ is a solution to Problem \ref{prob:original2}, $v^*(t)$ satisfies \eqref{eq:original constraints} almost everywhere. Then, from 2) of Assumption \ref{ass:psi}, we have 
    \begin{equation}
        \psi(u^*(t))=\psi(v^*(t))=0,
    \end{equation} 
    for almost all $t\in[0,T]$. Therefore, we have
    \begin{equation}
        \Omega(u^*)=\Omega(v^*)=0,\label{eq:Omega=0}
    \end{equation}
    from \eqref{eq:Omega}.
    Let $x(u;t)$ denote the state trajectory of the system \eqref{eq:plant_transformed} by control $u$ from $x(0)=\xi$.
    Then, it follows from the optimality of $u^*$ to Problem~\ikd{\ref{prob:relaxed_ver2}} 
    and \eqref{eq:Omega=0} that
    \begin{equation}
        x(u^*;T)^\top Q x(u^*;T) \leq x(v^*;T)^\top Qx(v^*;T).
        \label{eq:proof_leq}
    \end{equation}
    On the other hand, since $v^*$ minimizes the cost function in \ikd{\eqref{eq:cost_original}}  
    and $u^*$ satisfies \eqref{eq:original constraints} (i.e., $u^*$ is a feasible solution to Problem \ref{prob:original2}), we have
    \begin{equation}
        x(u^*;T)^\top Q x(u^*;T) \geq x(v^*;T)^\top Qx(v^*;T).
        \label{eq:proof_geq}
    \end{equation}
    Therefore, from \eqref{eq:proof_leq} and \eqref{eq:proof_geq}, we have
    \begin{equation}
        x(u^*;T)^\top Q x(u^*;T) = x(v^*;T)^\top Qx(v^*;T).
    \end{equation}
    That is, $u^*$ is also a solution to Problem \ref{prob:original2}.
\end{proof}

\section{\ikd{NUMERICAL COMPUTATION}} 
\label{sec:numerical computation}
In this section, we show a numerical method to solve Problem \ref{prob:relaxed_ver2}.
For this, we adopt time discretization \cite{Nag20}.

First, we divide the time interval
$[0,T]$ into $\ikd{K}$ 
subintervals as
\begin{equation}
[0,T] = [0,h) \cup [h,2h) \cup \cdots \cup [Kh-h,Kh],
\end{equation}
where $h>0$ is the sampling period and $K\in\mathbb{N}$
is the number of subintervals such that $T=Kh$.
On each subinterval, we assume the control signals $u_i(t)$, $i=1,2,\ldots, N$, are constant. Namely, $u_i(t)$ satisfies
\begin{equation}
 u_i(t) = \mathbf{u}_i[k], \quad t\in [kh, (k+1)h), \quad k=0,1,2,\ldots,K-1,
 \label{eq:zero-order-hold-control}
\end{equation}
where $\mathbf{u}_i[k]$ is the constant value of $i$-th control at time step $k$.
Such a control is called a \emph{zero-order-hold} signal, which is the output of a zero-order hold \cite{Nag20}.

Then, on each subinterval, the system \eqref{eq:plant_transformed} can be represented as
\begin{equation}
    \dot{x}(t) = \sum_{i=1}^N A_i \mathbf{u}_i[k] x(t),\quad t \in [kh,(k+1)h).
\end{equation}
By the Euler method, we have a discrete-time system
\begin{equation}
 \begin{split}
    \mathbf{x}[k+1] &= \mathbf{x}[\ikd{k}] + h \sum_{i=1}^N A_i \mathbf{u}_i[k] \mathbf{x}[k]\\ 
    &= \mathbf{x}[k] + h \Phi(\mathbf{x}[k])\mathbf{u}[k],
 \end{split}   
\end{equation}
where $\mathbf{x}[k]\triangleq x(kh)$ and
$\mathbf{u}[k] \triangleq [\mathbf{u}_1[k], \mathbf{u}_2[k], \ldots, \mathbf{u}_N[k]]^\top$.
The initial condition is given by $\mathbf{x}[0]=\xi$.
Next, the cost function is represented as
\begin{equation}
    \mathbf{x}[K]^\top Q\mathbf{x}[K] + \lambda  h\sum_{l=0}^{K-1} \psi(\mathbf{u}[l]).
\end{equation}
Finally, the constraints are given as
\begin{equation}
    0 \leq \mathbf{u}_i[k] \leq 1, \quad i = 1,2,\ldots,N,\quad
    \sum_{i=1}^N \mathbf{u}_i[k] = 1.
\end{equation}
We summarize the discretized optimization problem as follows.
\begin{problem}
    Find the control sequence $\mathbf{u}[k]\in\mathbb{R}^N$, $k=0,1,2,\ldots, \ikd{K-1}$, 
    that solves
    \begin{equation}
        \begin{aligned}
          & \underset{u}{\text{minimize}}
          & & \mathbf{x}[K]^\top \ikd{Q} \mathbf{x}[K] + \lambda h \sum_{l=0}^{K-1}\psi(\mathbf{u}[l])\\ 
          & \text{subject to}
          & & \mathbf{x}[k+1] =\mathbf{x}[k] +  h\Phi\bigl(\mathbf{x}[k]\bigr)\mathbf{u}[k],\quad
          \mathbf{x}[0]=\xi,\\
          & & & \mathbf{u}[k] \in\mathcal{U},\quad k=0,1,\ldots,K-1.
        \end{aligned}
    \end{equation}
  \end{problem}

Although the function $\psi$ is non-convex in general, the optimization problem can be numerically solved by the DC (difference of convex functions) programming \cite{Ike24}.
Also, we can use numerical optimization packages such as \texttt{cvxpy} \cite{DiaBoy16,AgrVerDiaBoy18} and \texttt{casadi} \cite{Casadi}.

\section{NUMERICAL EXAMPLES}
In this section, we show design examples of two- and three-mode switched systems.
We use the model predictive control (MPC) technique \cite{Mac} to realize feedback control with the proposed finite-horizon optimal control method.

\subsection{Example 1: Two-mode system}
Let us consider a two-mode switched system with
\begin{equation}
    A_1 =
    \begin{bmatrix}
        -3 & 1\\ 
        1 & 2
    \end{bmatrix},\quad
    A_2 = 
    \begin{bmatrix}
        1 & -0.5\\ 
        -3 & -5
    \end{bmatrix}.
\end{equation}
We note that $A_1$ and $A_2$ are unstable.
We take the initial state $x(0)=\xi=[-3,3]^\top$.

We adopt the regularization term $\Omega(u)$ as in \eqref{eq:regularization term}.
We simply take the weighting matrix $Q=I$ (the identity matrix).
We choose the regularization parameter $\lambda=1$, which is chosen by trial and error.
The sampling period for the Euler discretization is $h=0.1$.
The prediction horizon length for MPC is set to $K=10$.
With these parameters, we execute the simulation.

Figure \ref{2switch_control} shows the control signals $u_1(t)$ and $u_2(t)$ obtained by the proposed method. We can see that these signals satisfy the discreteness constraint in \eqref{eq:original constraints}.
\begin{figure}[t]
      \centering
      \includegraphics[width=1.0\linewidth]{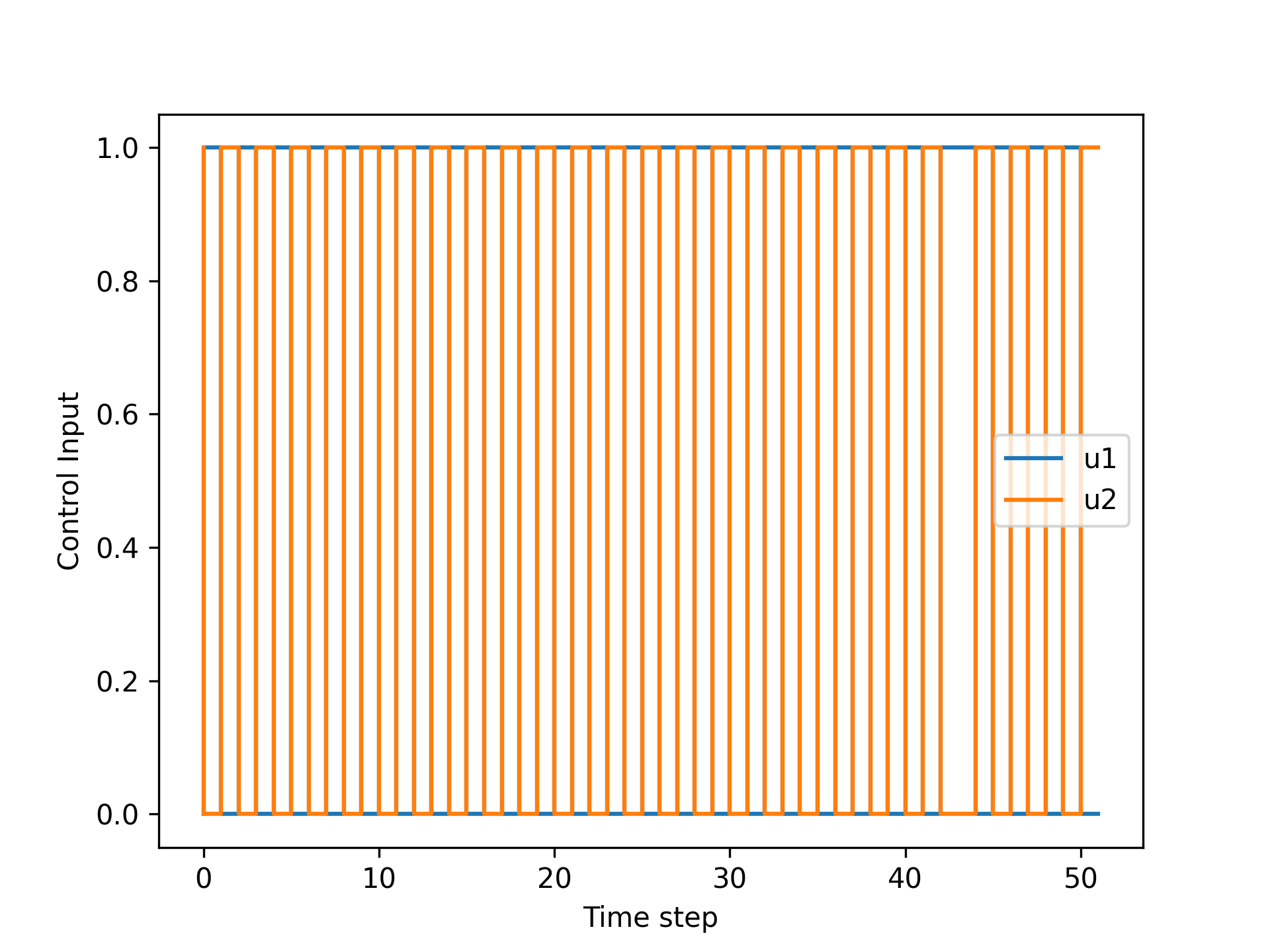}
      \caption{Control signals $u_1(t)$ and $u_2(t)$.}
      \label{2switch_control}
\end{figure}
Figure~\ref{2switch_state} shows the two states $x_1(t)$ and $x_2(t)$ under the obtained control. The states converge to the origin by the switching. Also, we show the state trajectory in the state space in Figure \ref{2switch_state-trajectory}.
\begin{figure}[t]
      \centering
      \includegraphics[width=1.0\linewidth]{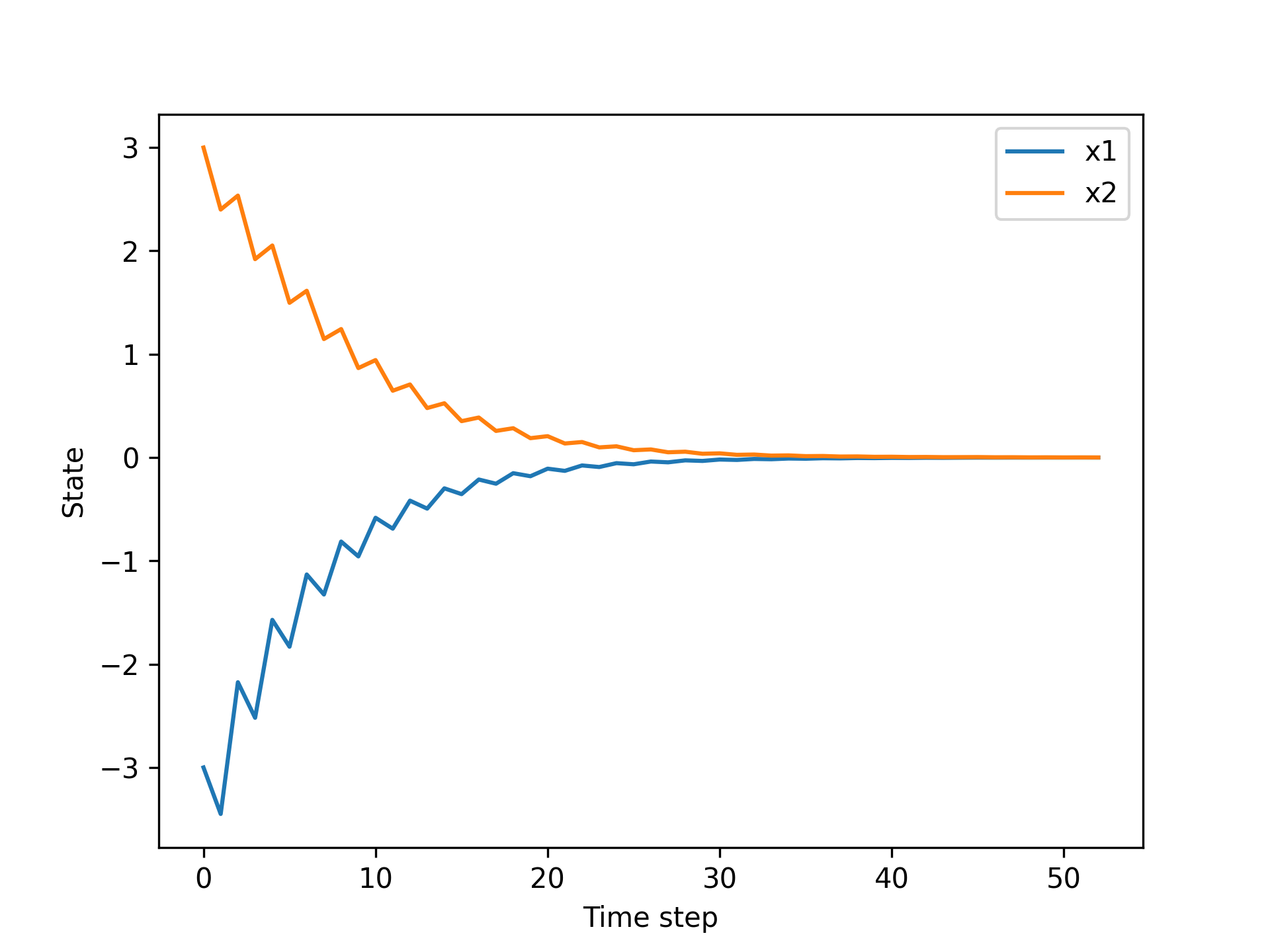}
      \caption{The states $x_1(t)$ and $x_2(t)$.}
      \label{2switch_state}
\end{figure}
\begin{figure}[t]
      \centering
      \includegraphics[width=1.0\linewidth]{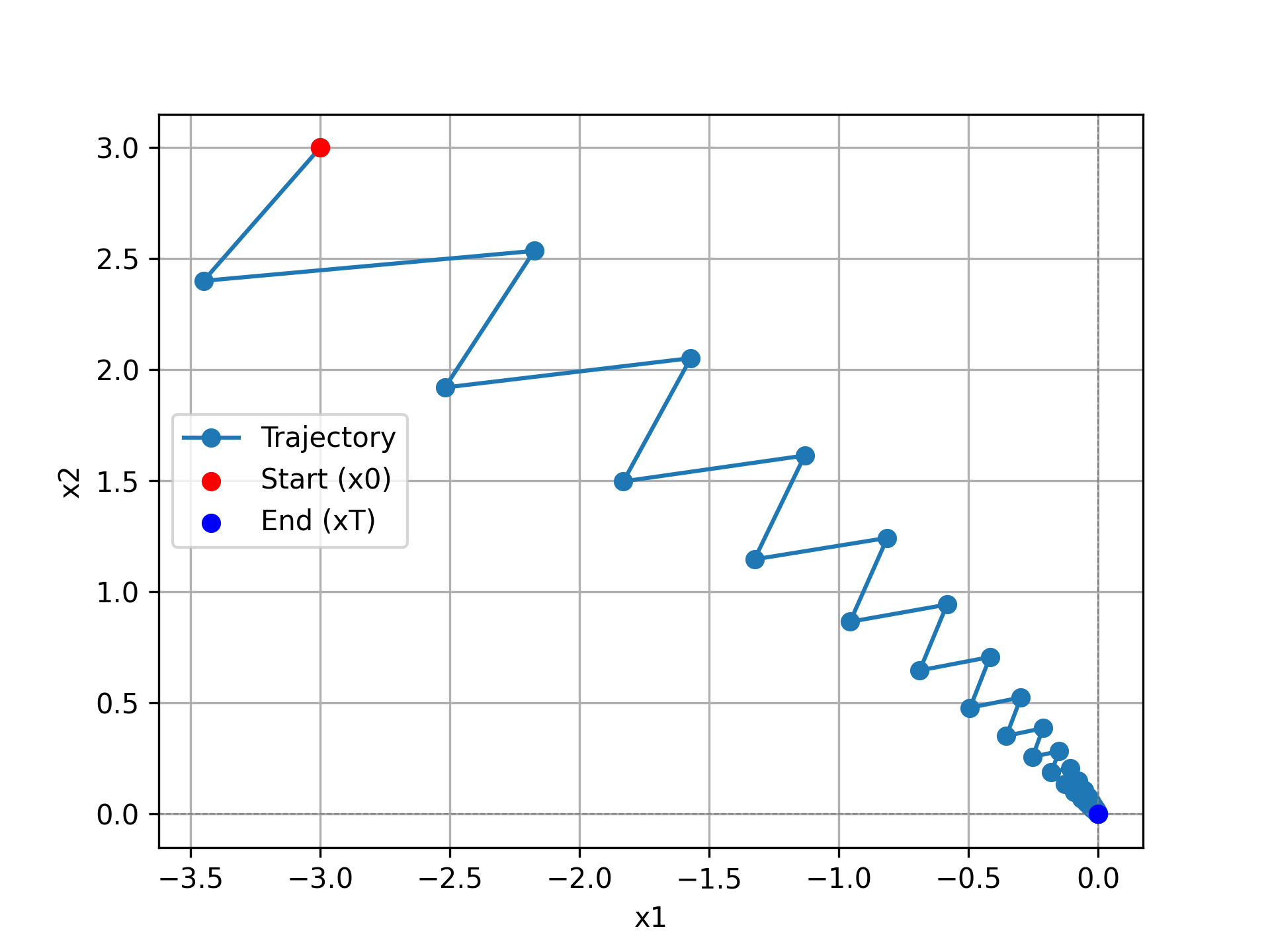}
      \caption{The state trajectory $(x_1(t),x_2(t))$ in the state space $\mathbb{R}^2$.}
      \label{2switch_state-trajectory}
\end{figure}
By these results, the proposed method successfully generates switching control signals that stabilize the unstable system.

\subsection{Example 2: Three-mode system}
We here consider a three-mode switched system with the following three unstable matrices:
\begin{equation}
    A_1 =
    \begin{bmatrix}
        -3 & 1\\ 
        1 & 2
    \end{bmatrix},\quad
    A_2 = 
    \begin{bmatrix}
        7 & 1\\ 
        7 & -15
    \end{bmatrix},\quad
    A_3 = 
    \begin{bmatrix}
        -5 & 2\\ 
        4 & 6
    \end{bmatrix}.
\end{equation}
We adopt \eqref{eq:regularization term} as the regularization term with the regularization parameter $\lambda=10000$, which is chosen by trial and error. For the other parameters, we take the same values as in Example 1. Namely, we take
$x(0)=\xi=[-3,3]^\top$, $Q=I$, $h=0.1$, and $K=10$.

Figure ~\ref{3switch_control} shows the control $u_1(t)$, $u_2(t)$, and $u_3(t)$, which we can see satisfy the discreteness constraint in \eqref{eq:original constraints}.
\begin{figure}[t]
      \centering
      \includegraphics[width=1.0\linewidth]{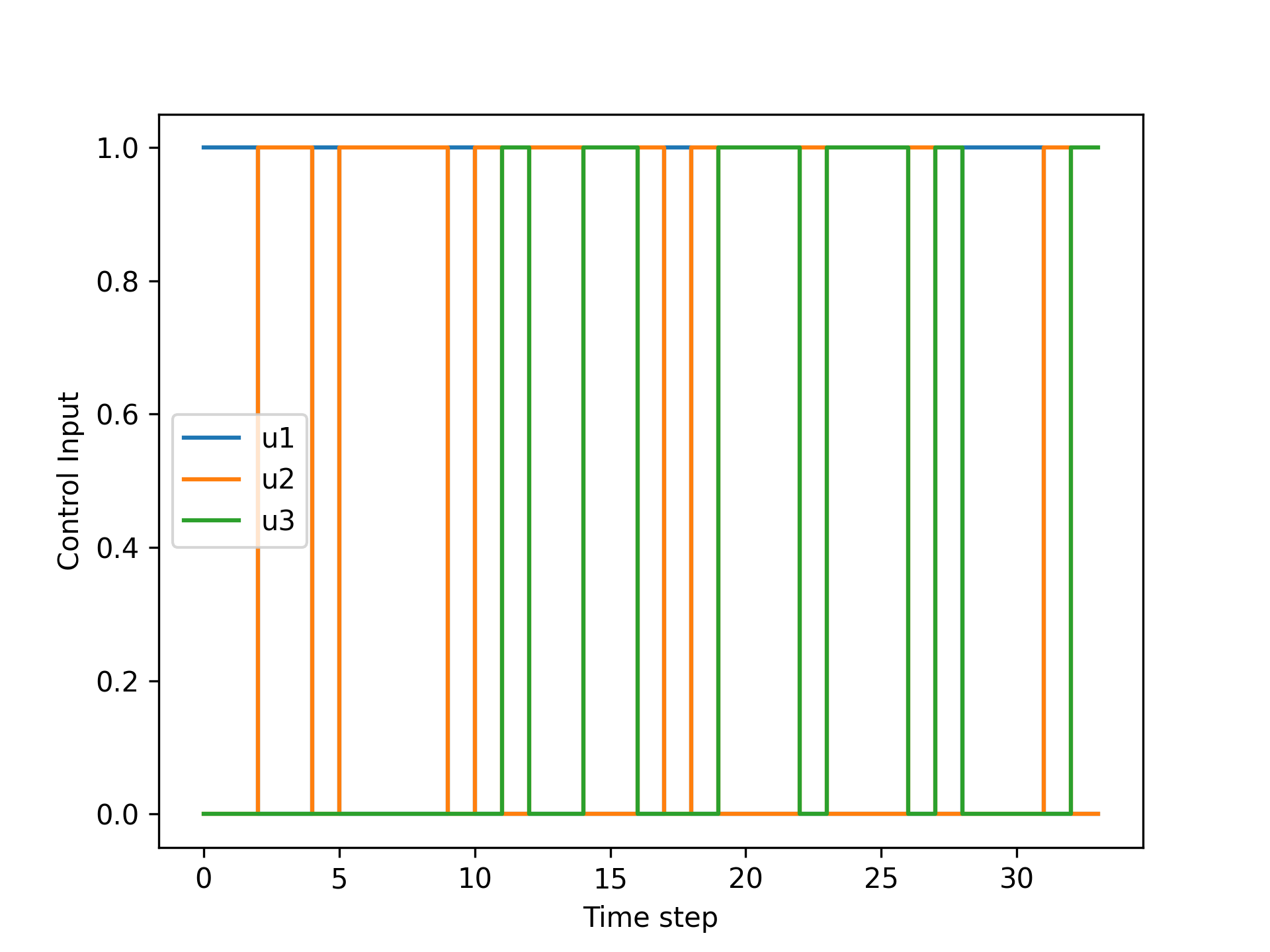}
      \caption{Control signals $u_1(t)$, $u_2(t)$, and $u_3(t)$.}
      \label{3switch_control}
\end{figure}
By this switching signal, the states $x_1(t)$ and $x_2(t)$ converge to the origin, as shown in Figure \ref{3switch_state}.
\begin{figure}[t]
      \centering
      \includegraphics[width=1.0\linewidth]{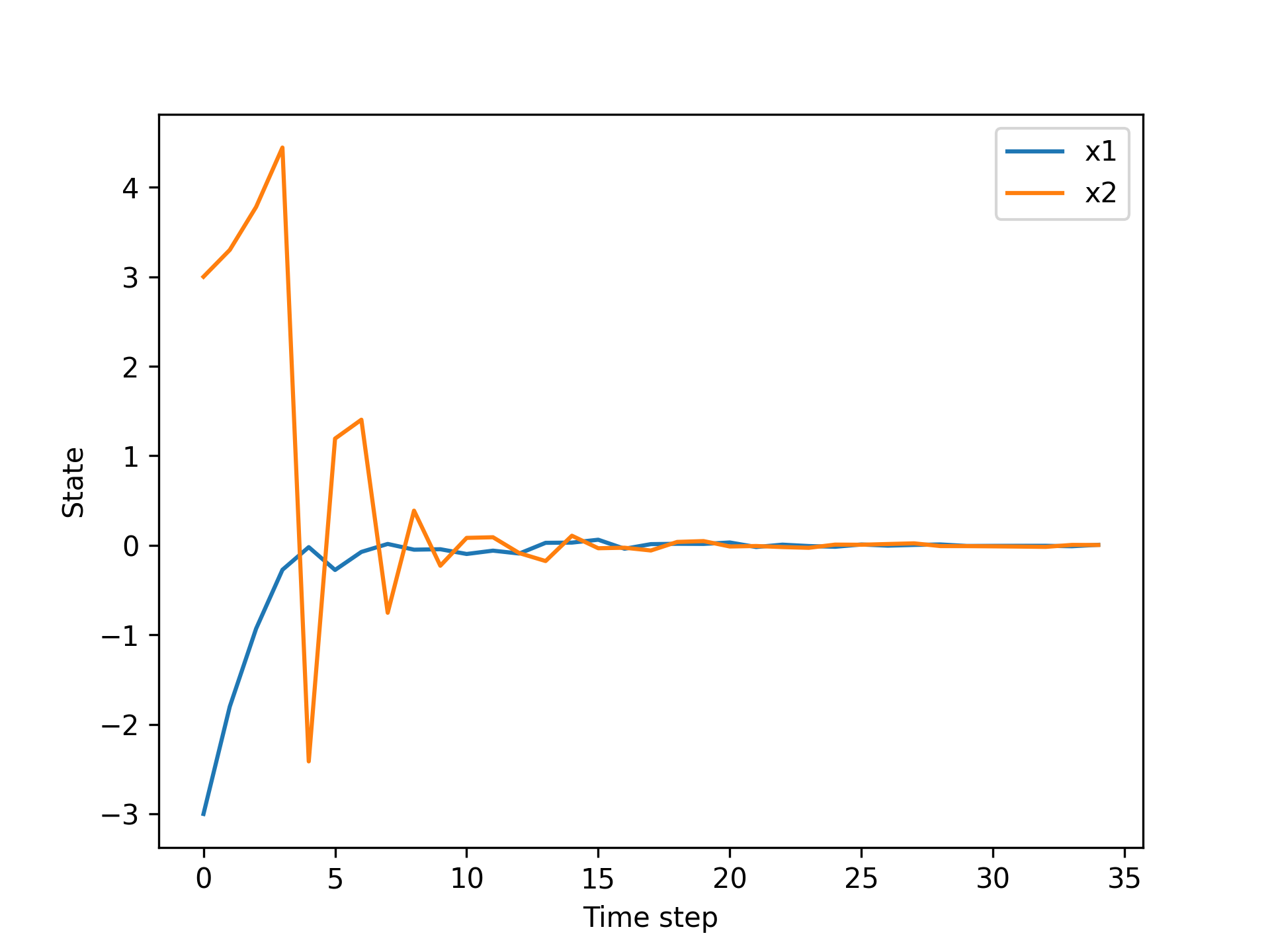}
      \caption{The states $x_1(t)$ and $x_2(t)$.}
      \label{3switch_state}
\end{figure}
The corresponding state trajectory in the state space is shown in Figure \ref{3switch_state-trajectory}.
\begin{figure}[t]
      \centering
      \includegraphics[width=1.0\linewidth]{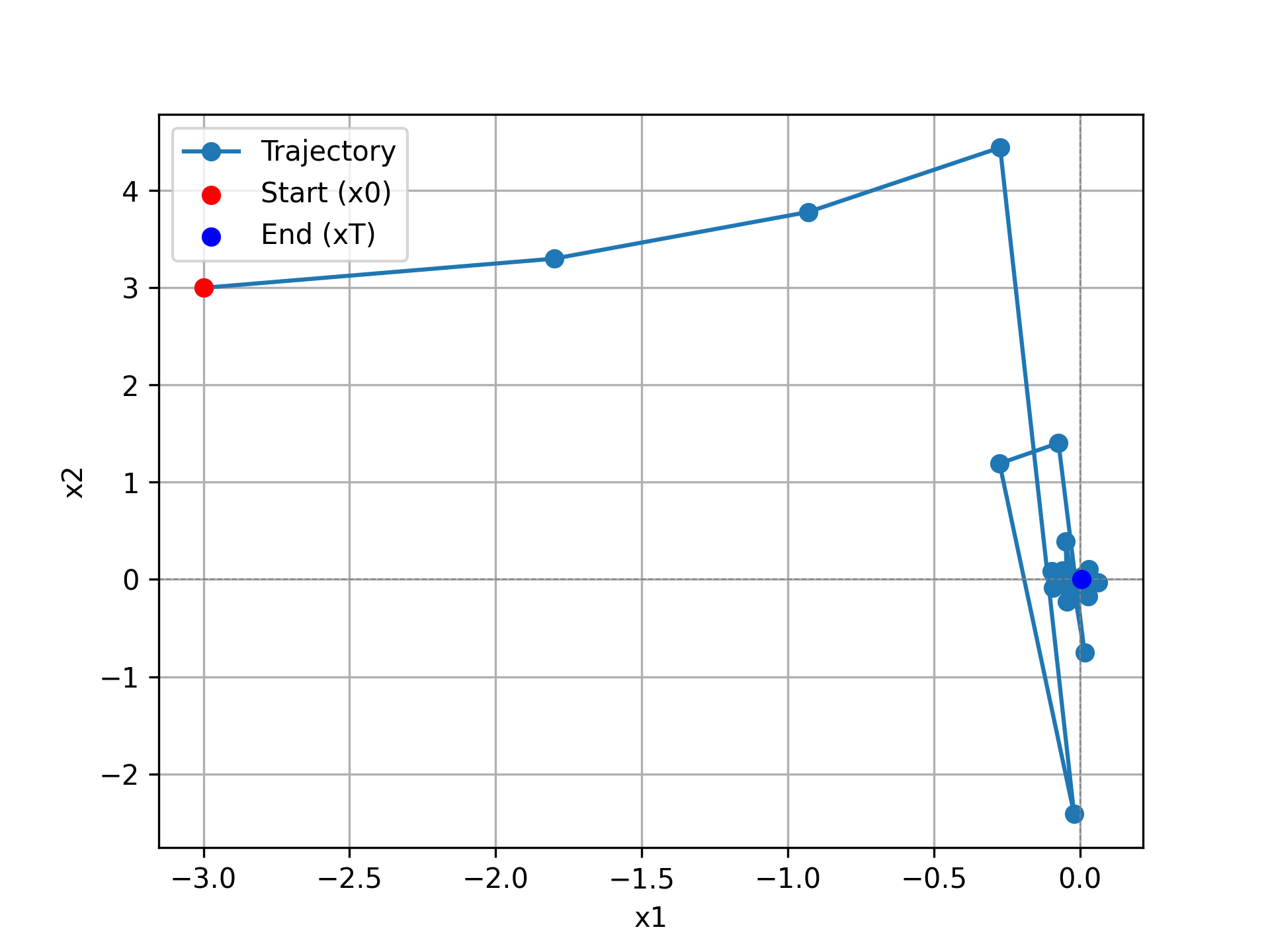}
      \caption{The state trajectory $(x_1(t),x_2(t))$ in the state space $\mathbb{R}^2$.}
      \label{3switch_state-trajectory}
\end{figure}
As shown in these figures, the proposed method is effective in designing switching control signals.

\section{CONCLUSION}
In this paper, we have proposed a design method of switching signals for linear time-invariant systems based on sparse optimization. Although the original problem is combinatorial optimization, we proposed a relaxed problem that can be efficiently solved. We have shown that any solution to the relaxed problem is also a solution to the original problem, and hence we can easily compute switching signals.
A merit of the proposed method is that the problem is described as an optimal control problem, and hence it is easily combined with a general constrained optimal control problem.  Future work includes the choice of the non-convex function $\psi(u)$ in Assumption \ref{ass:psi}, optimal control with state constraints, and 
the design of switching control signals for \emph{nonlinear} systems.



\bibliographystyle{IEEEtran}
\bibliography{IEEEabrv,reference-e}

\begin{thebibliography}{10}
\providecommand{\url}[1]{#1}
\csname url@samestyle\endcsname
\providecommand{\newblock}{\relax}
\providecommand{\bibinfo}[2]{#2}
\providecommand{\BIBentrySTDinterwordspacing}{\spaceskip=0pt\relax}
\providecommand{\BIBentryALTinterwordstretchfactor}{4}
\providecommand{\BIBentryALTinterwordspacing}{\spaceskip=\fontdimen2\font plus
\BIBentryALTinterwordstretchfactor\fontdimen3\font minus
  \fontdimen4\font\relax}
\providecommand{\BIBforeignlanguage}[2]{{%
\expandafter\ifx\csname l@#1\endcsname\relax
\typeout{** WARNING: IEEEtran.bst: No hyphenation pattern has been}%
\typeout{** loaded for the language `#1'. Using the pattern for}%
\typeout{** the default language instead.}%
\else
\language=\csname l@#1\endcsname
\fi
#2}}
\providecommand{\BIBdecl}{\relax}
\BIBdecl

\bibitem{LibMor99}
D.~Liberzon and A.~Morse, ``Basic problems in stability and design of switched
  systems,'' \emph{IEEE Control Systems Magazine}, vol.~19, no.~5, pp. 59--70,
  1999.

\bibitem{LinAnt09}
H.~Lin and P.~J. Antsaklis, ``Stability and stabilizability of switched linear
  systems: A survey of recent results,'' \emph{{IEEE} Trans. Autom. Control},
  vol.~54, no.~2, pp. 308--322, 2009.

\bibitem{HeTiaNiXu17}
R.~He, X.~Tian, Y.~Ni, and Y.~Xu, ``Mode transition coordination control for
  parallel hybrid electric vehicle based on switched system,'' \emph{Advances
  in Mechanical Engineering}, vol.~9, no.~8, p. 1687814017715564, 2017.

\bibitem{ZhaIfgPui24}
S.~Zhang, S.~Ifqir, and V.~Puig, ``Linear quadratic zonotopic control of
  switched systems: Application to autonomous vehicle path-tracking,''
  \emph{{IEEE} Contr. Syst. Lett.}, vol.~8, pp. 1895--1900, 2024.

\bibitem{GerColBol08}
J.~C. Geromel, P.~Colaneri, and P.~Bolzern, ``Dynamic output feedback control
  of switched linear systems,'' \emph{IEEE Transactions on Automatic Control},
  vol.~53, no.~3, pp. 720--733, 2008.

\bibitem{NogKoo20}
E.~Noghreian and H.~R. Koofigar, ``Power control of hybrid energy systems with
  renewable sources (wind-photovoltaic) using switched systems strategy,''
  \emph{Sustainable Energy, Grids and Networks}, vol.~21, p. 100280, 2020.

\bibitem{NiuZhaFanChe15}
B.~Niu, X.~Zhao, X.~Fan, and Y.~Cheng, ``A new control method for
  state-constrained nonlinear switched systems with application to chemical
  process,'' \emph{International Journal of Control}, vol.~88, no.~9, pp.
  1693--1701, 2015.

\bibitem{ZhuAnt15}
F.~Zhu and P.~J. Antsaklis, ``Optimal control of hybrid switched systems: A
  brief survey,'' \emph{Discrete Event Dynamic Systems}, vol.~25, pp. 345--364,
  2015.

\bibitem{IkeNagOno17}
T.~Ikeda, M.~Nagahara, and S.~Ono, ``Discrete-valued control of linear
  time-invariant systems by sum-of-absolute-values optimization,'' \emph{{IEEE}
  Trans. Autom. Control}, vol.~62, no.~6, pp. 2750--2763, 2017.

\bibitem{Nag15}
M.~Nagahara, ``Discrete signal reconstruction by sum of absolute values,''
  \emph{{IEEE} Signal Process. Lett.}, vol.~22, no.~10, pp. 1575--1579, Oct.
  2015.

\bibitem{Ike24}
T.~Ikeda, ``Nonconvex optimization problems for maximum hands-off control,''
  \emph{{IEEE} Trans. Autom. Control}, vol.~70, no.~3, pp. 1905--1912, 2025.

\bibitem{YinLouHeXin15}
P.~Yin, Y.~Lou, Q.~He, and J.~Xin, ``Minimization of $\ell_{1-2}$ for
  compressed sensing,'' \emph{SIAM Journal on Scientific Computing}, vol.~37,
  no.~1, pp. A536--A563, 2015.

\bibitem{IkeKas19}
T.~{Ikeda} and K.~{Kashima}, ``On sparse optimal control for general linear
  systems,'' \emph{{IEEE} Trans. Autom. Control}, vol.~64, no.~5, pp.
  2077--2083, 2019.

\bibitem{HayIkeNag24}
N.~Hayashi, T.~Ikeda, and M.~Nagahara, ``Design of sparse control with minimax
  concave penalty,'' \emph{IEEE Control Systems Letters}, 2024.

\bibitem{Nag20}
M.~Nagahara, \emph{Sparsity Methods for Systems and Control}.\hskip 1em plus
  0.5em minus 0.4em\relax Boston-Delft: now publishers, 2020.

\bibitem{DieBocDieWie06}
M.~Diehl, H.~G. Bock, H.~Diedam, and P.-B. Wieber, ``Fast direct multiple
  shooting algorithms for optimal robot control,'' in \emph{Fast motions in
  biomechanics and robotics: optimization and feedback control}.\hskip 1em plus
  0.5em minus 0.4em\relax Springer, 2006, pp. 65--93.

\bibitem{DiaBoy16}
S.~Diamond and S.~Boyd, ``{CVXPY}: {A} {P}ython-embedded modeling language for
  convex optimization,'' \emph{Journal of Machine Learning Research}, vol.~17,
  no.~83, pp. 1--5, 2016.

\bibitem{AgrVerDiaBoy18}
A.~Agrawal, R.~Verschueren, S.~Diamond, and S.~Boyd, ``A rewriting system for
  convex optimization problems,'' \emph{Journal of Control and Decision},
  vol.~5, no.~1, pp. 42--60, 2018.

\bibitem{Casadi}
J.~A.~E. Andersson, J.~Gillis, G.~Horn, J.~B. Rawlings, and M.~Diehl,
  ``{CasADi} -- {A} software framework for nonlinear optimization and optimal
  control,'' \emph{Mathematical Programming Computation}, vol.~11, no.~1, pp.
  1--36, 2019.

\bibitem{Cla}
F.~Clarke, \emph{Functional {A}nalysis, {C}alculus of {V}ariations and
  {O}ptimal {C}ontrol}, ser. Graduate Texts in Mathematics.\hskip 1em plus
  0.5em minus 0.4em\relax Springer, London, 2013, vol. 264.

\bibitem{Mac}
J.~M. Maciejowski, \emph{Predictive Control with Constraints}.\hskip 1em plus
  0.5em minus 0.4em\relax Prentice-Hall, 2002.

\end{thebibliography}

\end{document}